\documentclass[12pt]{amsart} 

\usepackage[english]{babel} 
\usepackage[all]{xy} 
\usepackage{amssymb} 
\usepackage{tikz}  

\numberwithin{equation}{section}

\newtheorem{thm}{Theorem}[section]
\newtheorem{lem}[thm]{Lemma} 
\newtheorem{prop}[thm]{Proposition} 
\newtheorem{cor}[thm]{Corollary} 

\theoremstyle{definition}
\newtheorem{nota}[thm]{Notation} 
\newtheorem{defin}[thm]{Definition}
\newtheorem{rem}[thm]{Remark}
\newtheorem{xamp}[thm]{Example} 

\def\Cor{Cor.\kern.2em} 
\def\Def{Def.\kern.2em} 
\def\Prop{Prop.\kern.2em} 
\def\Th{Th.\kern.2em} 

\def\eg{{e.g.}\kern.3em}
\def\ie{{i.e.}\kern.3em}
\def\resp{resp.\kern.3em}

\def\loccit{\textit{loc.\kern2pt cit.}\kern.25em}  
\def\loccitv{\textit{loc.\kern2pt cit.}} 

\def\calG{\mathcal{G}} 
\def\calR{\mathcal{R}} 

\def\DL{\mathrm{DL}}  
\def\dep{\mathrm{dep}} 
\def\haut{\mathrm{ht}} 
\def\val{\mathrm{val}} 
\def\wt{\mathrm{wt}}  

\let\vep\varepsilon

\def\sea{{\scriptscriptstyle\searrow}}
\def\nea{{\scriptscriptstyle\nearrow}}

\begin{document}

\title[Rooted trees and components of $L(\rho)\otimes L(\rho)$]{On some components of $L(\rho)\otimes L(\rho)$ 
associated with rooted trees for symmetrizable Kac-Moody algebras} 

\author{Rekha Biswal} 

\address{School of Mathematical Sciences, NISER, Bhubaneswar 752 050, Odisha, India,  and Homi Bhabha National Institute, Training School Complex, Anushakti Nagar, Mumbai 400094, India}

\email{rekha@niser.ac.in}

\author{Patrick Polo}

\address{Institut de Mathématiques de Jussieu\\
Sorbonne Université\\
4, place Jussieu -- Boîte Courrier 247\\
75252 Paris Cedex 05\\
France}

\curraddr{CNRS,  IRL 2000 ReLaX\\
Chennai Mathematical Institute
\\ H1 Sipcot IT Park, Siruseri, Kelambakkam 603 103
\\ Tamil Nadu, India} 

\email{patrick.polo@imj-prg.fr, ppolo@cmi.ac.in}

 \thanks{The second-named author is the corresponding author.}

 \thanks{Our collaboration started when the second-named author (P.P.) visited the first at NISER 
Bhubaneswar, during the period May 19 to June 6, 2026, and this paper was completed shortly after the 
first-named author (R.B.) visited the second at CMI, during the period June 8-13, 2026. 
Both authors thank both Institutions for their hospitality and propitious conditions of 
collaborative work.}

\thanks{R.B. is currently supported by the Anusandhan National Research Foundation (ANRF) through the MATRICS Grant ANRF/ARGM/2025/001683/MTR.}



\begin{abstract} 
Let $\mathfrak{g}$ be a symmetrizable Kac-Moody algebra over 
$\mathbb{C}$ and let $L(\rho)$ be the irreducible integrable 
$\mathfrak{g}$-module with highest weight $\rho$. 
Let $I$ be a subgraph of the Dynkin diagram of $\mathfrak{g}$ 
which has only simple bonds and no cycle of length $\geq 3$. 
For every subset $D$ of $I$, denote by $\beta_D$ the sum of 
the simple roots corresponding to $D$. To every $D \subset I$ 
such that $\lambda_{D,I} = 2\rho - \beta_I - \beta_D$ is dominant, 
we associate certain elements $\pi_{D,I}$ of weight 
$\lambda_{D,I} - \rho$ in the crystal $B(\rho)$, which 
depend on the choice of a root vertex in each connected component 
of $I$. Then we prove that our elements are $\rho$-dominant
 elements of $B(\rho)$, hence provide new families of components 
 of the tensor product  $L(\rho)\otimes L(\rho)$. 
\end{abstract} 

\date{\today} 

\maketitle

\section{Introduction} 

Let $\mathfrak{g}$ be a symmetrizable Kac-Moody algebra over $\mathbb{C}$, with Cartan subalgebra $\mathfrak{h}$. 
Let $A = (a_{st})_{s,t\in \Delta}$ be its Cartan matrix, where $\Delta$ is a finite set. We identify $\Delta$ 
with the set of simple roots $\{\alpha_s \}$ and denote the corresponding simple coroots by $\alpha_s^\vee$. 
Let $\leq$ denote the usual partial order on $\mathfrak{h}^*$, defined by $\lambda\geq \mu$ if 
$\lambda-\mu = \sum_{s \in \Delta} m_s \alpha_s$, with $m_s \in \mathbb{N}$. 
Let $P^+$ denote the set of dominant integral weights, \ie 
\begin{equation} 
P^+ = \{\lambda \in \mathfrak{h}^* \mid \langle \lambda, \alpha_s^\vee \rangle \in \mathbb{N}, \; \forall s \in \Delta\}.
\end{equation}
For each $\lambda \in P^+$, there exists a unique (up to isomorphism) integrable $\mathfrak{g}$-module $L(\lambda)$ with 
highest weight $\lambda$, and $L(\lambda)$ is irreducible, see \cite{Ka90}, \Cor 10.4. Further, by \loccitv, \Cor 10.7, 
the category of integrable $\mathfrak{g}$-modules is completely reducible, and for $\lambda, \mu\in P^+$ 
one has 
\begin{equation} 
L(\lambda)\otimes L(\mu) = \bigoplus_{\substack{\nu\in P^+\\[3pt] \nu \leq \lambda+\mu}} L(\nu)^{\oplus c_{\lambda,\mu}^\nu}
\end{equation} 
for some integers $c_{\lambda,\mu}^\nu \geq 0$. Let $\rho \in P^+$ such that 
$\langle \rho, \alpha_s^\vee\rangle = 1$ for all $s \in \Delta$. 

\smallskip Let $I$ be a subset of $\Delta$ such that $a_{ij} \in \{0,-1\}$ for $i\not= j$ in $I$ and such that 
$I$, regarded as a subgraph of the Dynkin diagram $\mathcal{D}$ of $\mathfrak{g}$, has no cycle of length $\geq 3$. 
These hypotheses imply that every connected component of $I$ is a tree. Set $N = \vert I\vert$. 

We introduce the following additional data $\calR$: we choose in each such tree $T$ 
a vertex called the root of $T$; this makes $T$ into  a rooted tree. 
We orient all the edges of $T$ towards the root; then if we have an edge $x \to y$ we say that 
$x$ covers $y$. 

Given this data $\calR$, we consider a numbering $i_1,\dots, i_N$ of the elements of $I$ 
which has the property that $i_k \to i_\ell$ implies $k > \ell$, and we consider the following 
product of lowering operators:  
\begin{equation} 
f_I^\sea (\calR) = f_{i_N} \cdots f_{i_1} 
\end{equation} 
where the symbol $\sea$ indicates that the product is taken in decreasing order of the indices. Further, we number arbitrarily $v_{N+1}, \dots, v_{\vert \Delta\vert}$ 
the elements of $\Delta - I$.

Then, to every subset $D$ of $I$ we associate the integral weight 
\begin{equation} 
\lambda_{D,I} = 2\rho - \sum_{i\in I} \alpha_i - \sum_{d\in D} \alpha_d 
\end{equation} 
and the following product of lowering operators:  
\begin{equation} 
f_D^\nea (\calR) = \prod_{d\in D} f_d 
\end{equation} 
where this time the product is taken in increasing order of the indices (whence the symbol $\nea$). 
The products $f_I^\sea (\calR)$ and $f_D^\nea (\calR)$ depend only on $\calR$ 
(see Remark \ref{rem-prod-R} and Lemma \ref{prod-by-depth}). 

Let $B(\rho)$ denote the crystal of $\mathfrak{g}$ with highest weight $\rho$ and let $\pi_\rho$ be its element of 
weight $\rho$. Recall that a non-zero element $x\in B(\rho)$ is called $\rho$-dominant if 
$\wt(x) + \rho$ is dominant and $\vep_s(x) \leq 1$ for all $s\in \Delta$. 
Our main result is the following theorem,  which is a partial extension of Proposition 2.3 of \cite{BG25}. 

\begin{thm}\label{th-1} Let $D$ be a subset of $I$ such that $\lambda_{D, I}$ is dominant. 
Then the element 
$\pi_{D, I}(\calR) = f_D^\nea (\calR) f_I^\sea (\calR) \pi_\rho$ 
is non-zero, of weight $\lambda_{D,I} - \rho$, and is $\rho$-dominant. 
\end{thm} 

Since the $\rho$-dominant elements of $B(\rho)$ are in bijection with the components of the tensor product 
$L(\rho) \otimes L(\rho)$ (see \cite{Ka95}, \Prop 4.2 or \cite{Ja96}, \Prop 9.27), this gives the following corollary. 

\begin{cor}\label{cor-1} Let $D$ be a subset of $I$ such that $\lambda_{D, I}$ is dominant. 
Then $L(\lambda_{D,I})$ occurs with multiplicity $\geq 1$ in $L(\rho)\otimes L(\rho)$. More precisely, 
a lower bound for this multiplicity is the number 
$n(D,I)$ of distinct elements $\pi_{D, I}(\calR)$ obtained as $\calR$ ranges over all possible choices 
of root vertices. 
\end{cor}

We do not know how  to give in general an explicit lower bound for $n(D,I)$. However we worked out several  examples in the last section of the paper.

Our proof combines crystal theory with combinatorics of rooted trees and proceeds in two steps. 
In more details, for $k = 1,\dots, \vert \Delta\vert$, let $B_k$ be the crystal with 
highest weight the fundamental weight $\omega_k$ and let $\pi_k$ denote the element of $B_k$ of weight 
$\omega_k$. We consider $B(\rho)$ as the connected component of 
the crystal 
$B_1\otimes \cdots \otimes B_{\vert \Delta\vert}$ containing the element $\pi_\rho = 
 \pi_1 \otimes \cdots \otimes \pi_{\vert \Delta\vert}$. 

Firstly, in Proposition \ref{prop-1}, 
we describe explicitly,  in terms of some decomposition of $D$ into 
``\emph{stems with dead leaves}'', the element $\pi_{D, I}$, showing in particular that it 
 is non-zero. 
 
Secondly, in Proposition \ref{prop-2}, we prove, 
using the previous description and \Prop 2.3 of \cite{BG25}, 
 that $\vep_s(\pi_{D, I}) \leq 1$ for all $s\in \Delta$.

\section{Preliminaries} 

\subsection{Tensor product of crystals} 
For the tensor product of crystals, we follow the convention of \cite{BS17}, \S 2.3 (which is opposite to that in 
Kashiwara's original papers). In particular, we will use Lemma 2.33 of \loccit (see also \cite{KN94}, \Prop 2.1.1 or 
\cite{Ka93}, Lemma 1.3.6), 
which we recall below for the 
convenience of the reader. Firstly, let us introduce the relevant notation. 

\begin{nota} (a) If $x$ is an element of a crystal $B$, then for each $s\in \Delta$ we set $\wt_s(x) = 
 \langle \wt(x), \alpha_s^\vee \rangle$. 
 
\smallskip (b) Let $x_j \in B_j$, for crystals $B_1, \dots, B_n$. 
For all $s\in \Delta$ and $j = 1,\dots,n$, we set: 
\begin{equation} 
\Sigma_s(x,j) = \sum_{h=1}^{j-1} \wt_s(x_h), \qquad \quad 
\Sigma'_s(x,j) = -\sum_{h=j+1}^{n} \wt_s(x_h). 
\end{equation} 
\end{nota} 

Then Lemma 2.33 of  \cite{BS17} is the  following: 

\begin{lem}\label{lem-BS} 
Let $x_j \in B_j$, for crystals $B_1, \dots, B_n$. Consider the element $x = x_1 \otimes \cdots \otimes x_n$ of 
$B_1 \otimes \cdots \otimes B_n$. Then, for all $s\in \Delta$, 

\smallskip {\rm (a)} $\varphi_s(x)$ is given by 
$\varphi_s(x) = \max_{1\leq j\leq n} \Big( \varphi_s(x_j) + \Sigma_s(x,j) \Big)$ 
and if  $j$ is the first index where the maximum is attained, then
$$
f_s(x)
=
x_1 \otimes \cdots \otimes f_s(x_j)\otimes \cdots \otimes x_n.
$$

{\rm (b)} $\vep_s(x)$ is given by 
$\vep_s(x) = \max_{1\leq j\leq n} \Big( \vep_s(x_j) + \Sigma'_s(x,j) \Big)$
and if  $j$  is the last index where the maximum is attained, then
$$
e_s(x)
=
x_1 \otimes \cdots \otimes e_s(x_j)\otimes \cdots \otimes x_n.
$$
\end{lem} 

\begin{nota} For the sake of brevity, we will denote $\varphi_s(x_j) + \Sigma_s(x,j)$ by $\Phi_s(x,j)$ and 
$\vep_s(x_j) + \Sigma'_s(x,j)$ by $\Upsilon_s(x,j)$. 
\end{nota}

\subsection{Recollections about trees} Let us recall a few facts about trees and rooted trees. 

\begin{nota} If $\Gamma$ is a  graph (unoriented) and $v$ a vertex of $\Gamma$, the \emph{valency}  of $v$, 
denoted by $\val_\Gamma(v)$ or simply $\val(v)$, is the number of its neighbours. 
\end{nota} 

\begin{nota} Let $T$ be a tree. If we choose, once for all, a vertex $v_0$ of $T$ called the \emph{root}, and 
 orient all the edges towards the root, then $T$ becomes a \emph{rooted tree}. 
For each vertex $v$, let 
$\calG_v$ be the unique geodesic joining $v$ to $v_0$; then 
$\calG_v$ has the  form 
$$
\xymatrix@1{ v= v_h \ar[r] & v_{h-1} \ar[r] & \cdots \ar[r] & v_0}
$$
for some $h$ and we define the \emph{height} of $v$ as $\haut(v) = h = \# \calG_v - 1$. 
We set $\haut(T) = \max_v \haut(v)$. 

\smallskip If $v\not= v_0$, 
\ie if $h\geq 1$, we denote the vertex $v_{h-1}$ by $v^-$ and call it the \emph{predecessor} of $v$. 
We will also write $x\to y$ to mean that $y = x^-$, and say in this case that $x$ \emph{covers} $y$. 

\smallskip The vertices $v$ such that $v\not= v_0$ and $\val(v) = 1$ are called the \emph{leaves} of $T$. 
For any vertex $v$, we define its \emph{depth}, denoted by $\dep_T(v)$ or simply $\dep(v)$, as the maximum 
cardinality of a geodesic from a leaf $\ell$ of $T$ to $v$. 
Thus, in particular, $\dep(v) = 0$ if and only if $v$ is a leaf, and one has $\dep(v_0) = \haut(T)$. 

\smallskip The previous notions extend immediately to the case of a forest of rooted trees: if $I$ is such a forest, 
and $v$ is a vertex of $I$, one has $\haut_I(v) = \haut_T(v)$ and $\dep_I(v) = \dep_T(v)$, where $T$ is the connected component of $I$ containing $v$. 
\end{nota} 

\subsection{A certain numbering of $\mathcal{D}$}\label{compat-order} 
Let us come back to the Dynkin diagram $\mathcal{D}$ of $\mathfrak{g}$. 
Let $I$ be a subset of $\Delta$ such that $a_{ij} \in \{0,-1\}$ for  $i\not= j$ in $I$ and such that 
$I$, regarded as a subgraph of $\mathcal{D}$, has no cycle of length $\geq 3$. 
These hypotheses imply that every connected component of $I$ is a tree. 

\begin{defin}\label{def-order} 
In each connected component $T$ of $I$, we choose a vertex as the root of $T$. This makes $I$ 
into a forest of rooted trees; in particular, $I$ has a height function $\haut_I$. 
We fix then a numbering  $v_1,\dots, v_N$ of the vertices of $I$, where $N = \vert I \vert$, which satisfies the following two conditions: 
\begin{enumerate} 
\item It is compatible with the height function, \ie $\haut_I(v_j) > \haut_I(v_i)$ implies $j > i$. 

\item For each given height, the elements of $D$ come last, \ie if $v_j \in D$ and $v_i \in I - D$ 
have the same height $h$, then $i < j$.
\end{enumerate} 
This can be obtained by numbering $v_1,\dots, v_d$ the elements of $I - D$ of height $0$ (if any), 
then $v_{d+1},\dots, v_n$ the elements of $D $ of height $0$ (if any), then 
$v_{n+1},\dots, v_{n+r}$ the elements of $T - D$ of height $1$ (if any), and so on. 

Finally, we number arbitrarily $v_{N+1}, \dots, v_{\vert \Delta\vert}$ the elements of 
$\Delta - I$. 
\end{defin}

\begin{defin}\label{def-pi} (a) For $k = 1,\dots, \vert \Delta\vert$, let $B_k$ be the crystal with 
highest weight the fundamental weight $\omega_k$ and let $\pi_k$ denote the element of $B_k$ of weight 
$\omega_k$. 

\smallskip (b) 
Let $\pi =  \pi_1 \otimes \cdots \otimes \pi_{\vert \Delta\vert}$; this is the element of highest weight $\rho$ in the crystal 
$B_1\otimes \cdots \otimes B_{\vert \Delta\vert}$ and it belongs to the subcrystal $B(\rho)$. 

\smallskip (c) Let 
$\pi_I = f_I^\sea \, \pi = \Big(\prod_{i\in I} f_i \Big) \pi$, 
where the product is taken in \emph{decreasing} order of the indices, \eg if $I = \{1,3,4\}$ then 
$\pi_I = f_4 f_3 f_1 \pi$. (This explains the notation $f_I^\sea$.) 
\end{defin} 

\begin{lem}\label{lem-pi-I} 
One has $\pi_I = \Big( \bigotimes_{i=1}^N f_i \pi_i \Big) \otimes \Big( \bigotimes_{j > N} \pi_j \Big)$, recalling that 
$N = \vert I \vert$. 
\end{lem}

\begin{proof} By induction on $\vert I\vert$. There is nothing to prove if $I = \varnothing$, so we may assume 
$I \not= \varnothing$
and the result proved for $I' = I - \{N\}$, where $N$ is the largest element of $I$. 
Set $X' = \pi_{I'}$. We want to compute $f_N X'$. 

\smallskip 
By the inductive hypothesis, one has $x_j = \pi_j$ if $j\geq N$ and $x_j = f_j \pi_j$ if $j < N$. 
From this it follows that $\varphi_N(x_j) = \delta_{N,j} = \wt_N(x_j) $ for $j\geq N$, while for $j < N$ one has: 
$\varphi_N(x_j) = 1 = \wt_N(x_j)$ if $j = N^-$ and $\varphi_N(x_j) = 0 = \wt_N(x_j)$ else. 

Thus, if $N$ has a predecessor $N^-$ one has: 
$$
\Phi_N(x,j) = 
\begin{cases} 
0 & \text{if } j < N^-, 
\\
1 & \text{if } N^- \leq j < N, 
\\
2 & \text{if } j \geq  N,
\end{cases}
$$
and otherwise one has $\Phi_N(x,j) = 0$ if $j < N$ and $\Phi_N(x,j) = 1$ if $j\geq N$. 
In both cases, $N$ is the first index where the maximum of $\Phi_N(x,j)$ is attained. Hence, by Lemma \ref{lem-BS} (a), 
one obtains, with obvious notation: 
$$
f_N X' = X'_{\leq N-1} \otimes f_N \pi_N \otimes X'_{> N}.
$$
This proves the lemma.
\end{proof} 

\begin{rem}\label{rem-prod-R} In the proof, we used only the fact that $j\to i$ implies $i < j$. This shows that 
$\pi_I$ depends only on the data $\calR$. 
\end{rem} 

\section{Description of the element $\pi_{D, I}$} 

\subsection{Preliminaries} 

\begin{defin}\label{def-inc} 
Let $f_D^\nea = \prod_{d\in D} f_d$, where now the product is taken in \emph{increasing} order 
of the indices, \eg if $D = \{1,3\}$ then $f_D^\nea = f_1 f_3$.  (This explains the notation $f_D^\nea$.) 

Set $\pi_{D,I} = f_D^\nea \, \pi _I$ and note that, if $\pi_{D, I}$ is non-zero, then $\wt(\pi_{D,I} ) + \rho$ is the weight 
\begin{equation}
\lambda_{D,I} =  2\rho - \sum_{i\in I} \alpha_i - \sum_{d\in D} \alpha_d 
\end{equation} 
considered in the Introduction. Then, regarding $D$ as a subgraph of $I$, one has the following lemma. 
\end{defin} 

\begin{lem}\label{lem-dom} 
$\lambda_{D,I} $ is dominant if and only if for each $d\in D$ one of the following conditions is satisfied: 
\begin{enumerate} 
\item[(a)] $d$ has at least two neighbours in $I$, or 

\smallskip 
\item[(b)] $d$ has a unique neighbour $i$ in $I$, and $i\in D$. 
\end{enumerate} 
\end{lem} 

\begin{proof} One sees immediately that for any simple root $\alpha_j \in \Delta$, one has 
$$
\begin{cases} 
\langle\lambda_{D,I}, \alpha_j^\vee \rangle  \geq 2 & \text{ if } j\not\in I , 
\\
\langle\lambda_{D,I}, \alpha_j^\vee \rangle  \geq 0 & \text{ if } j\in I - D , 
\\
\langle\lambda_{D,I}, \alpha_j^\vee \rangle  = \val_I(j) + \val_D(j) - 2 & \text{ if } j \in D , 
\end{cases} 
$$
where $\val_I(j)$, \resp $\val_D(j)$, denote the number of neighbours of $j$ in $I$, \resp in $D$. 
Thus, $\lambda_{D,I} $ is dominant if and only if for every $d\in D$ one has 
$$
 \val_I(d) + \val_D(d) \geq 2. \leqno (*)
$$
This condition is automatically satisfied if $\val_I(d) \geq 2$. Otherwise, since $\val_D(d) \leq \val_I(d)$, 
one must have $\val_I(d) = 1 = \val_D(d)$, \ie has a unique neighbour $i$ in $I$, and $i\in D$. 
This proves the lemma. 
\end{proof} 

Note that the lemma implies, in particular, that an isolated point of $I$ (\ie a connected component  
of cardinality $1$), cannot belong to $D$. 

\begin{rem}\label{commut-2A1} We shall repeatedly use the following commutativity property: 
for $s\not= t$ in $\Delta$ which are not neighbours in the Dynkin diagram, 
 the action of the lowering operators $f_s$ and $f_t$ on any highest weight crystal commute 
(and similarly for the raising operators $e_s$ and $e_t$). 
This is evident in Kashiwara's definition in terms of quantized enveloping algebras. 

On the other hand, Stembridge has taken this property, together 
with crystal analogues of the Serre relations in the case where $s$ and $t$ are linked by a simple bond, 
to define the class of what is now called Stembridge crystals, and he proved that any highest weight crystal 
$B(\lambda)$ satisfies these properties, see \cite{St03}, \Def 1.1 and \Th 2.4 (see also \cite{BS17}, \S\S 4.2 
and 5.7). 
\end{rem}

Note that $D$, as a (full) subgraph of the forest $I$, is also a forest, \ie its connected components 
$D_1,\dots, D_u$ are trees. Further,  the total order on $I$, compatible with the covering relation $x\to y$, 
induces on each $D_t$ a structure of rooted tree. 
In particular, $D$ has a height $h = \haut(D)$ and a depth function $\dep_D$. 

\begin{defin} Recall the notation $f_D^\nea $ from {\rm \ref{def-inc}}. For each connected component 
$D_t$ of $D$, set $f_{D_t}^\nea = \prod_{d\in D_t} f_d$, 
where the product is taken in increasing order of the indices. Note that $f_d$ and $f_{d'}$ commute if 
$d$ and $d'$ belong to different components, by Remark \ref{commut-2A1}. 
Thus, one obtains that: 
\begin{equation}\label{inc-comp}
f_D^\nea = \prod_{t= 1}^u f_{D_t}^\nea 
\end{equation} 
where the product on the right can be taken in any order. 
\end{defin}

\begin{defin} For $k= 0,\dots, h$, set $D[k] = \{d \in D \mid \dep_D(d) = k\}$ and consider 
the operator
\begin{equation} 
f_{D[k]} = \prod_{d \in D[k]} f_d. 
\end{equation} 
In this product the $f_d$'s are pairwise non-linked, hence they commute (see Remark \ref{commut-2A1}), 
so that the order of the product does not matter. 
\end{defin}

\begin{lem}\label{prod-by-depth}  
One has $f_D^\nea  = f_{D[h]} \cdots f_{D[0]}$. 
\end{lem}

\begin{proof} Suppose first that $D$ is a rooted tree, of height $h$, and let us prove the result by induction on $h$. 
The result is true for $h=0$ since the factors of 
$f_{D[0]}$ commute with each other. So we may assume $h\geq 1$. Let $d_0$ be  the root of 
$D$ and set $D' = D -\{d_0\}$. One has $f_D^\nea = f_{d_0} f_{D'}^\nea$ and, since 
$\haut(D') = h-1$ and $\dep_{D'}(d) = \dep_{D}(d)$ for every $d\in D'$, one obtains by the inductive hypothesis 
that 
$$
 f_{D'}^\nea = f_{D'[h-1]} \cdots f_{D'[0]} = f_{D[h-1]} \cdots f_{D[0]} . 
$$
This proves the result in the case of a rooted tree. Thus, in  the general case, the result is true for each connected 
component $D_t$ of $D$, \ie for every $t$ one has 
$$
f_{D_t}^\nea = f_{D_t[h]} \cdots f_{D_t[0]} 
$$
with the convention that $f_{D_t[k]} = 1$ if $k$ is greater than the height of $D_t$. Again, since $f_d$ and $f_{d'}$ commute if $d$ and $d'$ belong to different components, the above equality combined with \eqref{inc-comp} 
gives that 
$$
f_D^\nea = \prod_{k = h}^0 \Big( \prod_{t=1}^u f_{D_t[k]} \Big) = f_{D[h]} \cdots f_{D[0]}. 
$$
This proves the lemma. 
\end{proof}

\subsection{More notation about $D$} 

In order to describe conveniently the element $\pi_{D,I}$, we need to introduce more definitions. 

\begin{defin} For every $d\in D$, let $C_I(d) =  \{i\in I \mid i \to d\}$ 
be the set of elements of $I$ which cover $d$, and let $c_I(d) = \vert C_I(d)\vert$. We also set $C_D(d) = 
D \cap C_I(d)$. 

Note that if $d$ is a leaf of $D$ then either $C_I(d) = \varnothing$ (in which case $d$ is also a leaf of $I$), or all elements of $C_I(d)$ belong to $I - D$. 
\end{defin} 

\begin{defin} Let $d$ be a leaf of $D$. 

\smallskip (a) If $c_I(d) \geq 2$, we say that $d$ is a \emph{vivid leaf}.  

\smallskip 
(b) Otherwise, if $c_I(d)\leq 1$, we say that $d$ is a \emph{dead leaf} of $D$. In this case, by Lemma \ref{lem-dom}, 
$d$ has a predecessor $d^-$ in $I$ and, further, $d^- \in D$ if $c_I(d) = 0$. 

If $d^- \in D$, we will say that $d^-$ is a \emph{receiving node}, and if $d^-\not\in D$ we will say that 
$d^-$ is a \emph{grounding point}. (We shall see in \Prop \ref{prop-1} that in $\pi_{D, I}$ all dead leaves fall, in some sense, 
on their predecessor.)

\smallskip (c) We denote by $\DL$ the set of dead leaves of $D$ and by $\DL^n$, \resp $\DL_g$,  
the subset consisting of those $d$ such that $d^-\in D$, \resp $d^-\not\in D$. 
(Thus, $\DL_g$, \resp $\DL^n$, is the set of dead leaves which fall on the ground, \resp on a receiving node.) 
These two subsets form a partition of $\DL$.  Further, for $i \in I$ we set 
$\DL(i) = \{d \in \DL \mid d \to i\}$. 

\smallskip (d) We set $D^* = D - \DL$ and $D^n = D - \DL_g = D^* \cup \DL^n$. 
\end{defin} 

\begin{lem}\label{lem-def-m} There exists a unique function $\sigma : D^* \to I$ which is defined recursively as follows: 
\begin{enumerate} 
\item If $\dep_D(d) = 0$, \ie if $d$ is a vivid leaf, $\sigma(d)$ is 
the largest element of $C_I(d)$. Then we say that $\sigma(d)$ is an \emph{anchor point}. 

\smallskip 
\item If $\dep_D(d) = 1$ and $C_D(d)$ contains only dead leaves, $\sigma(d)$ is the largest element of 
$C_D(d)$. Then we say that $\sigma(d)$ is an \emph{anchor dead leaf}. 

\smallskip 
\item In all other cases, the set $U(d)$ of $d' \in C_D(d)$ such that $\sigma(d')$ is  defined  is not empty, and the values 
$\sigma(d')$ are pairwise distinct. Then $\sigma(d) = d^+$, where $d^+$ is the unique element of $U(d)$ such that 
$\sigma(d^+)$ is maximal. 
\end{enumerate}
\end{lem}

\begin{proof} By design, one has $\sigma(d)^- = d$ at each step, and this ensures that $\sigma$ is injective. Therefore, 
in situation (3), the values $\sigma(d')$, for $d' \in U(d)$, are indeed pairwise distinct, and hence $\sigma$ can be defined 
recursively on the whole of $D^*$. 
\end{proof} 

\begin{defin} For $i\in I$, its \emph{precursors} are the elements of $I$ defined by $i^0 = i$ and 
$i^{k+1} = (i^k)^-$ for $k\in \mathbb{N}$. Thus, the set $\mathcal{P}(i)$ of precursors of $i$ is the stem 
$$
i = i^0 \to i^1 \to \cdots \to i^h
$$
where $i^h$ is the root of the tree containing $i$ and $h$ is the height of $i$. 
\end{defin} 

\begin{defin} (a) For $d\in D^*$, let $t$ be the largest integer such that $\sigma^t(d)$ is defined, 
where $\sigma^t$ is the $t$-fold iterate of $\sigma$. 
The set $\{d,\sigma(d),\dots, \sigma^t(d)\}$ is called the \emph{gradient flow} through $d$. We denote its 
end-point by $\mu(d)$; it is an element of $I$ which is either an anchor point or an anchor dead leaf.

\smallskip 
(b) Let $M$ be the set of those $\mu(d)$'s and let $m_1 < \cdots < m_t$ be the elements of $M$ arranged in increasing order. 
For $r = 1,\dots, t$,  define the \emph{gradient line} 
\begin{equation}
L_r = \{m_r\} \cup \{d \in \mathcal{P}(m_r) \mid \mu(d) = m_r\} 
\end{equation}
and denote by $\ell_r$ the smallest element of $L_r$, \ie $L_r$ is the stem: 
\begin{equation}\label{Lr-stem} 
m_r \to m_r^1 \to \cdots \to m_r^h = \ell_r. 
\end{equation}
Further, it is convenient to extend the function $\mu$ to $D^* \cup M$ by setting $\mu(m) = m$ for $m \in M - D$, 
\ie if $m$ is an anchor point.  

\smallskip (c) For $r = 1,\dots, t$, let $S_r$ be the union of $L_r$ and the dead leaves $d\in \DL^n$ such that 
$d^- \in L_r$. We say that $S_r$ is a ``\emph{stem with dead leaves at the receiving nodes}''. 
\end{defin} 

\begin{lem}\label{lem-stems} 
{\rm (a)} For each $r$, one has $\ell_r = m_r^h$, where $h$ is the largest integer such that
 $\mu(m_r^h) = m_r$, and one has $\mu(m_r^k) > m_r$ for $k > h$ if $m_r^k$ is defined. 

\smallskip {\rm (b)}  $S_1,\dots, S_t$ form a partition of $D^n \cup M$.
\end{lem} 

\begin{proof} (a) Note first that all elements of $\mathcal{P}(m_r) - \{m_r\}$ belong to $D^*$. If $d,d'\in D^*$ 
and $d\to d'$ then, by the definition of $\sigma$, one has $\mu(d)\leq \mu(d')$. Assertion (a) follows from this. 

\smallskip (b) By assertion (a), the $L_k$ form a partition of $D^* \cup M$, since there are the level lines of the function $\mu$. Further, each dead leaf $d \in D^n$ fall onto its predecessor $d^-$, which belongs to $D^*$ 
hence to a unique $L_k$. Thus $d$ belongs to $S_k$ and to no other $S_\ell$. This proves (b). 
\end{proof} 

\begin{rem} Finally, the dead leaves $b\in \DL_g$ are the elements $b\in D$ such that $b^- = i$ and 
$C_I(b) = \{j\}$ for some $i,j \in I - D$, \ie one has $j \to b \to i$ and $j$ is the unique element of $I$ which covers $b$. 
Note, firstly,  that there can be several dead leaves $b_1,\dots, b_p\in \DL_g$ sharing the same ground point $i$ and, 
secondly, that $j$ cannot be one of $m_1,\dots, m_t$, but $i$ could be one of them. This leads us to the following definition. 
\end{rem} 

\begin{defin} (a) Denote by $\calG$ the subset of $I$ consisting of the ground points which do not belong to $M$. For $i\in \calG$, set $G_i = \{i\} \cup  \DL(i)$ and call it a $0$-stem of dead leaves. 

\smallskip 
(b) On the contrary, suppose that $m_r$ is both a ground point (for the dead leaves $b_1,\dots, b_p$) and an 
anchor point (for the stem $S_r$). In this case, we set $C_r = S_r \cup \{b_1,\dots, b_p\}$ 
and say that $C_r$ is a \emph{corolla with dead leaves}, see the picture in the example \ref{xamp-GDM} below. 
We set also $\mathcal{C} = \calG \cap M$ and say that the elements of $\mathcal{C}$ are \emph{corolla points}. 
\end{defin} 

Using the previous definition, one deduces easily from Lemma \ref{lem-stems} the following corollary. 

\begin{cor} The $G_i$, for $i\in \calG$, and the corollas or stems with dead leaves $C_s$ or $S_r$ for $s\in \mathcal{C}$ 
or $r\in M -\mathcal{C}$, form a partition of the subset $D \cup M \cup \calG$ of $I$. 
\end{cor} 

\begin{defin} (a) We set $f_{\DL(i)} = \prod_{j \in \DL(i)} f_j$ for each 
$i\in I$ such that $\DL(i)\not= \varnothing$. Since the $f_j$'s are not linked, 
they commute (see Remark \ref{commut-2A1}) and hence the order of the product does not matter.

\smallskip 
(b) For $r=1,\dots, t$, we set $L_r^* = L_r - \{m_r\}$ and define the operator 
$f_{L_r} = \prod_{d\in L_r^*} f_d$, 
where the product is taken in increasing order of the indices, \ie using the notation \eqref{Lr-stem}, one has 
\begin{equation}\label{def-f-Lr} 
f_{L_r} = f_{\ell_r} \cdots f_{m_r^1} . 
\end{equation} 
\end{defin} 

\subsection{The main proposition} 

Set 
\begin{equation} 
\pi_{D,I} = y_1 \otimes \cdots \otimes y_{\vert \Delta\vert} .
\end{equation} 
Then the $y_k$'s are described by the following proposition, which says, informally, that 
for each dead leaf $d$ the operator $f_d$ falls down to the predecessor of $d$, 
whereas the product $f_{L_r}$ of the operators in each gradient line moves up to the terminal anchor of that line.

\begin{prop}\label{prop-1} Recall that the elements of $M$ are $m_1 < \cdots < m_t$. Then, one has: 
$$
y_k = 
\begin{cases} 
f_{\DL(k)} f_{L_r} f_k \pi_k & \text{if $k = m_r$ for } r = 1,\dots, t; 
\\
f_{\DL(k)} f_k \pi_k & \text{if $k$ is a receiving node or $k\in \calG$}; 
\\
 f_k \pi_k & \text{else}. 
\end{cases}
$$
where in the first case, $f_{\DL(m_r)} = 1$ if $m_r$ is an anchor point which is not a corolla point. 
\end{prop} 

Before proving this proposition, let us give an example.

\begin{xamp}\label{xamp-GDM}  Consider the following tree $I$, where $a = 10$, \dots, $h = 17$, and the elements of $D$ are indicated in bold: 
$$
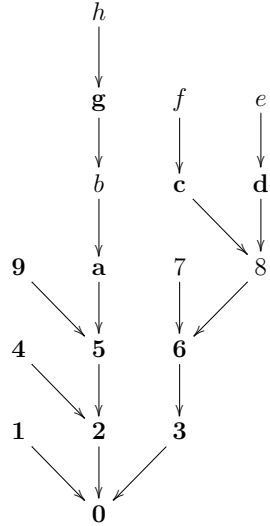
\scalebox{.8}{
\xymatrix{
& h \ar[d] & & 
\\
& {\bf g} \ar[d] & f \ar[d] & e \ar[d] 
\\
& b \ar[d] & {\bf c} \ar[dr]  & {\bf d} \ar[d] 
\\
 {\bf 9}   \ar[rd] & {\bf a}  \ar[d] & 7  \ar[d] & 8 \ar[ld] 
\\
{\bf 4} \ar[rd] & {\bf 5} \ar[d]  & {\bf 6} \ar[d] 
\\
{\bf 1} \ar[rd] & {\bf 2} \ar[d]  & {\bf 3} \ar[ld] 
\\
 & {\bf 0} & 
}}
$$
Here, $c,d,g$ are the dead leaves in $\DL_g$, the other dead leaves are $1,4,9,a$, and 
$6$ is a vivid leaf. 
One has $M = \{8,a\}$ and $8$ is a corolla point. Further, one has $\calG = \{b\}$ and the partition of $D\cup M\cup 
\calG$ is: 
$$
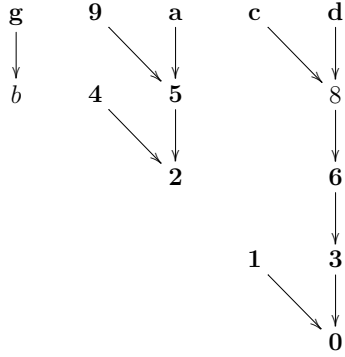
\scalebox{.8}{
\xymatrix{
{\bf g} \ar[d] \\ b } \qquad 
\xymatrix{ 
 {\bf 9}   \ar[rd] & {\bf a}   \ar[d] 
\\ 
{\bf 4} \ar[rd] & {\bf 5} \ar[d] 
\\
& {\bf 2} 
}
\qquad 
\xymatrix{ 
{\bf c} \ar[dr]  & {\bf d} \ar[d] 
\\
& 8  \ar[d] 
\\
& {\bf 6} \ar[d] 
\\
{\bf 1} \ar[rd] & {\bf 3} \ar[d] 
\\
 & {\bf 0} 
}
}
$$
Then  Proposition \ref{prop-1} says that the elements $y_k$ for $k\in D\cup M\cup \calG$ are given as follows, where the terms coming from 
$f_D^\nea$ are indicated in bold:  
$$
\begin{array}{|c|} 
\hline 
f_g \pi_g
\\[2pt] 
{\bf f_g} f_b \pi_b
\\
\hline 
\end{array} 
\qquad 
\begin{array}{|cc|} 
\hline 
f_9 \pi_9 & {\bf f_2 f_5} f_a \pi_a
\\[2pt] 
f_4 \pi_4 & {\bf f_9 f_a} f_5 \pi_5  
\\[2pt] 
& {\bf f_4} f_2 \pi_2 
\\
\hline 
\end{array} 
\qquad 
\begin{array}{|cc|} 
\hline 
f_c \pi_c  & f_d \pi_d 
\\[2pt] 
& ({\bf f_0 f_3 f_6})( {\bf f_c f_d}) f_8 \pi_8  
\\[2pt] 
& f_6 \pi_6 
\\[2pt] 
f_1 \pi_1 & f_3 \pi_3 
\\[2pt] 
 & {\bf f_1} f_0 \pi_0 
\\
\hline 
\end{array} 
$$
\end{xamp}

\subsection{Proof of \Prop  \ref{prop-1}} We are going to prove the result by induction on the depth of 
$D$. Further, using \eqref{inc-comp}, we may restrict ourselves, when convenient, to the case where $D$ is 
a tree. Note that if $k = m_r$ is a corolla point, the commuting factors $f_{L_r}$ and $f_i$, for $i\in \DL(k)$, will 
be obtained separately, as coming from different connected components of $D$. 

\smallskip Recall that $\pi_I = x_1 \otimes \cdots \otimes x_{\vert \Delta\vert}$, with $x_i = f_i \pi_i$ 
for $i \leq N = \vert I \vert$ and $x_k = \pi_k$ for $k> N$. For each $i\in I$, one has 
\begin{equation} 
\wt(f_i \pi_i) = -\omega_i + \omega_{i^-} + \sum_{\substack{s\in \Delta\\ s\to i}} \omega_s
\end{equation} 
where the term $\omega_{i^-}$ has to be omitted if $i^-$ does not exist. Therefore, for every $d\in D$, one has 
\begin{equation}\label{eqref-depth0} 
\begin{cases} 
\phi_d(x_d) = 0, \; \wt_d(x_d) = -1 & 
\\
\phi_d(x_k) = 1 = \wt_d(x_k) & \text{if $k \in C_I(d)$ or } k = d^- \in I , 
\\
\phi_d(x_k) = 0 = \wt_d(x_k) & \text{else}. 
\end{cases} 
\end{equation} 

\subsubsection{}\label{subsub-depth0} Firstly, let $d$ be a leaf of $D$. 

\smallskip (1-a) Suppose that $d$ is a vivid leaf. Set $\delta = 1$ if $d^-$ exists and belongs to $I$, and 
$\delta = 0$ else. Let $\ell, \dots, m = \sigma(d)$ be the elements of $C_I(d)$, arranged in increasing order. 
Their number is $c_I(d) \geq 2$. 
Then the values of $\phi_d, \wt_d, \Sigma_d$ and $\Phi_d$ at $d, \ell, \dots, m$ are as shown in 
the following table:  
$$
\begin{array}{|c|c|c|c|c|c|}
\hline 
 & f_d \pi_d & f_\ell \pi_\ell & \cdots & f_m \pi_m   
\\
\hline 
\varphi_d & 0 & 1 & \cdots & 1 
\\
\hline 
\wt_d & -1 & 1 & \cdots & 1 
\\
\hline 
\Sigma_d & \delta & \delta- 1 & \cdots & \delta + c_I(d)  - 2  
\\
\hline 
\Phi_d & \delta & \delta & \cdots & \delta + c_I(d)  - 1  
\\
\hline
\end{array}
$$
Further, by \eqref{eqref-depth0}, one has $\Phi_d(\pi_I,k) = c_I(d) - 1 + \delta$ for every $k > m$. 
Since $c_I(d) - 1 + \delta > \delta$, it follows from Lemma \ref{lem-BS} that when one applies 
$f_d$ to $\pi_I$, the factor $f_d$ is attached at the $m$-th place, as asserted in \Prop \ref{prop-1}.  

\smallskip (1-b) Suppose now that $d$ is a dead leaf, \ie $\delta = c_I(d)$ equals $0$ or $1$. 
Then, by Lemma \ref{lem-dom},  $i = d^-$ exists and belongs to $I$ (and $i\in D$ if $\delta =0$). 
If $\delta = 1$, let $\ell$ denote the unique element of $C_I(d)$. 
Then the values of $\phi_d, \wt_d, \Sigma_d$ and $\Phi_d$ at $d, \ell, \dots, m$ are as shown in the 
 following table, where the last column exists only if $\delta = 1$: 
$$
\begin{array}{|c|c|c|c|c|c|}
\hline 
 & f_i \pi_i & f_d \pi_d & f_\ell \pi_\ell  
\\
\hline 
\varphi_d & 1 & 0 & \delta  
\\
\hline 
\wt_d & 1 & -1 & \delta  
\\
\hline 
\Sigma_d & 0 &1 & 0  
\\
\hline 
\Phi_d & 1 & 1 & \delta  
\\
\hline
\end{array}
$$
Further, by \eqref{eqref-depth0}, one has $\Phi_d(\pi_I,k) = \delta$ for every $k > m$. 
Thus, it follows from Lemma \ref{lem-BS} that when one applies 
$f_d$ to $\pi_I$, the factor $f_d$ is attached at the $i$-th place, as asserted in \Prop \ref{prop-1}.  
This proves the proposition when $D$ has depth $0$.

\subsubsection{}\label{subsub-depth1}  Now, set $X^1 = f_{D[0]} \pi_I$ and write $X^1 = x_0^1 \otimes \cdots \otimes x_{\vert \Delta\vert}^1$. 
By the first step, we know that 
\begin{equation}\label{eqref-X1} 
x_k^1 = 
\begin{cases} 
f_{\DL(k)} f_k \pi_k & \text{if $k$ is a ground point or receiving node}, 
\\
f_v f_k \pi_k & \text{if $k = \sigma(v)$ for a vivid leaf } v, 
\\
x_k & \text{else}. 
\end{cases} 
\end{equation} 
Let $d\in D$ such that $\dep_D(d) = 1$. Set $\delta = 1$ if $d^-$ exists and belongs to $I$, and $\delta = 0$ 
else. 
Suppose that $C_I(d)$ contains dead leaves $b_1,\dots, b_q$, vivid leaves $v_1,\dots, v_p$, and elements 
$i_1,\dots, i_r$ of $I - D$, all arranged in increasing order. Since $\dep_D(d) = 1$ one has $p+q \geq 1$. 
Set $(q-1)_+ = \max(q{-}1,0)$ and $c'_I(d) = q + c_I(d)$. 

\smallskip Case (2-a). Suppose first that $p\geq 1$. Set $m = \sigma(v_p) = \mu(d)$. 
Then one has the following table, where the relative order of the $b_k$'s and the $v_\ell$'s does not matter, 
since all have height $< \haut(m)$, hence all are $< m$: 
$$
\begin{array}{|c|c|c|c|c|c|} 
\hline 
 &  f_{\DL(d)} f_d \pi_d & f_{i_1} \pi_{i_1} &  \cdots &  f_{v_p} \pi_{v_p} & f_{v_p} f_m \pi_m
\\
\hline 
\varphi_d & (q{-}1)_+ & 1 & \cdots & 1 & 1  
\\
\hline 
\wt_d & q{-}1 & 1 & \cdots & 1 & 1 
\\
\hline 
\Sigma_d & \delta & \delta {-1 +} q   & \cdots & \delta {-2 +} c'_I(d) & \delta {-1 +} c'_I(d) \rule{0pt}{10.5pt}
\\
\hline 
\Phi_d & \delta{+} (q{-}1)_+ & \delta {+} q   & \cdots &  \delta {-1 +} c'_I(d) & \delta {+} c'_I(d) \rule{0pt}{10.5pt}
\\
\hline
\end{array}
$$
and for $k > m$ one has $\varphi_d(x_k^1) = 0 = \wt_d(x_k^1)$, hence 
 the value of $\Phi_d$ remains constant for $k \geq m$. 
 Since $p\geq 1$, the maximum value of $\Phi_d$ is 
 $\delta{+}c'_I(d)$, which is attained for the first time at the $m$-th place. Thus,  
it follows from Lemma \ref{lem-BS} that when one applies 
$f_d$ to $X^1$, the factor $f_d$ is attached at the $m$-th place, where 
$m$ is the anchor point $\mu(d)$, as asserted in \Prop \ref{prop-1}.  

\smallskip Case (2-b). Suppose now that $p = 0$. Then $q \geq 1$ and $\sigma(d) = b_q$. 
Further, $c'_I(d) = 2q + r$. 
By the choice of our ordering in \Def \ref{def-order}, 
$i_1,\dots, i_r$ are smaller than $b_1,\dots, b_q$. Then, one has the following table: 
$$
\begin{array}{|c|c|c|c|c|} 
\hline 
 &  f_{\DL(d)} f_d \pi_d & f_{i_1} \pi_{i_1} &  \cdots &  f_{b_q} \pi_{b_q} 
\\
\hline 
\varphi_d & q{-}1 & 1 & \cdots & 1   
\\
\hline 
\wt_d & q{-}1 & 1 & \cdots & 1  
\\
\hline 
\Sigma_d & \delta & \delta {-1 +} q   & \cdots & \delta {-2 +} c'_I(d)   \rule{0pt}{10.5pt}
\\
\hline 
\Phi_d & \delta{+} q{-}1 & \delta {+} q   & \cdots &  \delta {-1 +} c'_I(d)  \rule{0pt}{10.5pt}
\\
\hline
\end{array}
$$
and for $k > m$ one has $\varphi_d(x_k^1) = 0 = \wt_d(x_k^1)$, hence 
 the value of $\Phi_d$ remains constant for $k \geq b_q$. Since $q\geq 1$, the maximum value of $\Phi_d$ is 
 $\delta{+}c'_I(d)$, which is attained for the first time at the place $b_q$. Thus,  
it follows from Lemma \ref{lem-BS} that when one applies 
$f_d$ to $X^1$, the factor $f_d$ is attached at the $b_q$-th place, where 
$b_q$ is the anchor dead leaf $\mu(d)$. 
This proves the proposition when $D$ has depth $\leq 1$. 

\subsubsection{}\label{subsub-depth2}  Finally, assume that $d\in D$ satisfies $\dep_D(d) = u \geq 2$ and that the proposition is 
proved up to depth $u-1$. Set 
\begin{equation} 
X^u = f_{D[u-1]} \cdots f_{D[0]} \pi_I 
\end{equation} 
and write $X^u = x_1^u \otimes \cdots \otimes x^u_{\vert \Delta \vert}$. 

As before, we set $\delta = 1$ if $i = d^-$ exists and belongs to $I$, and $\delta = 0$ else. 
Suppose that $C_I(d)$ contains elements $i_1,\dots, i_r$ of $I - D$, 
dead leaves $b_1,\dots, b_q$, vivid leaves $v_1,\dots, v_p$, all arranged in increasing order, and  
elements $d_1,\dots, d_t$ of $D$ of depth $\geq 1$, arranged so that 
$\mu(d_1) < \cdots < \mu(d_t)$ (if $t\geq 2$). 
Since $\dep_D(d) \geq 2$, one has $t\geq 1$ and each $\mu(d_s)$ is defined. 

\begin{lem} 
One has $\dep_D(d_t) = u-1$ and $\mu(d) = \mu(d_t)$. 
\end{lem} 

\begin{proof} 
Since $\dep_D(d) = u$  there exists at least one $d_s$ such that $\dep_D(d_s)$ $= u-1$. Then 
$\mu(d_s)$ has height $\haut(d)+u$ if it is an anchor dead leaf, and $\haut(d)+u+1$ if it is an anchor point. 
If one had $\dep_D(d_t) \leq u-2$, then $\mu(d_t)$ would have height $\leq \haut(d)+u$, with equality only 
if $\mu(d_t)$ is an anchor point. But by our choice of ordering \ref{def-order}, the anchor points of a given height 
are smaller than the anchor dead leaves of the same height. This proves that  $\dep_D(d_t) = u-1$. 

Similarly, each $\mu(v_j)$, for $j = 1,\dots, p$, is an anchor point of height $\haut(d) + 2$, hence is smaller 
than $\mu(d_t)$. Therefore, $\mu(d) = \mu(d_t)$. 
\end{proof}

Then, by the inductive hypothesis, we have 
\begin{equation}\label{eqref-depthu} 
\begin{cases} 
\phi_d(x_{d^-}^u) = 1 = \wt_d(x_{d^-}^u) & \text{if $d^-$ exists and is in } I, 
\\
\phi_d(x_d^u) = (q{-}1)_+, \; \wt_d(x_d) = q{-}1 & 
\\
\phi_d(x_k^u) = 1 = \wt_d(x_k) & \text{if } k\in C_I(d), 
\\
\phi_d(x_k^u) = 1 = \wt_d(x_k) & \text{if } k \in \mu\big( \{v_1,\dots, d_t\} \big), 
\\
\phi_d(x_k^u) = 0 = \wt_d(x_k^u) & \text{else}. 
\end{cases} 
\end{equation} 
Note that for $j$ running through the ordered set  $R(d) = \{i, d\} \cup C_I(d) \cup \{\mu(v_1), \dots, \mu(v_p)\} \cup 
\{\mu(d_1), \dots, \mu(d_t)\}$, the function $\Phi_d$ can fail to increase only for $j = i$. Then 
each element $k > d$ of $R(d)$ satisfies $\varphi_d(x^u_k) = 1 = \wt_d(x^u_k)$, hence adds $1$ to the sum 
$\Sigma_d$. 

Since $t\geq 1$, one deduces from this that $\Phi_d(X^u,k)$ attains its maximum value  for the first time at the place 
$\mu(d)$ (and remains constant for $k\geq \mu(d)$). Thus,  
it follows from Lemma \ref{lem-BS} that when one applies 
$f_d$ to $X^u$, the factor $f_d$ is attached at the place corresponding to 
 the anchor point or anchor dead leaf $\mu(d)$. This completes the proof of \Prop \ref{prop-1}.  
 \hfill $\square$

\section{Completion of the proof of Theorem \ref{th-1}} 

In this section we will prove the following: 

\begin{prop}\label{prop-2} Let $D$ be a subset of $I$ such that $\lambda_{D,I}$ is dominant. 
Then $\pi_{D,I}$ is a $\rho$-dominant element of the crystal $B(\rho)$, \ie 
one has $\vep_s(\pi_{D,I}) \leq 1$ for all $s\in \Delta$.
\end{prop} 

The proof will occupy the rest of this section. 
Firstly, note that if $j\not\in D$, then $2\alpha_j + \wt(\pi_{D, I})$ is not $\leq \rho$, hence not a weight of $B(\rho)$. Therefore, $e_j^2 \, \pi_{D,I} = 0$, whence $\vep_j(\pi_{D, I}) \leq 1$. Thus, it suffices to prove that $\vep_d(\pi_{D, I}) \leq 1$ for every $d\in D$. 

\smallskip 
Recall the notation $\pi_{D, I} = y_1 \otimes \cdots \otimes y_{\vert \Delta \vert}$. Let $d\in D$. One deduces from 
\Prop \ref{prop-1} that $\wt_d(y_k) < 0$ if and only if: 
\begin{enumerate} 
\item $k=d$ and $y_d = f_d \pi_d$, or 

\smallskip 
\item $d = \ell_r$ for some $r=1,\dots, t$ and $k = \mu(d)$, or 

\smallskip
\item $d\in \DL(k)$, 
\end{enumerate} 
and, in these cases, $\wt_d(y_k) = -1$. Moreover, using  \cite{BG25}, \Prop 2.3, one sees that 
$\vep_d(y_k) = 1$ in these cases, while $\vep_d(y_k) = 0$ in all other cases. 
Note further that cases (2) and (3) cannot occur together, since in case (3) $d$ is a dead leaf, while it is not 
in case (2). 
One deduces from this that: 
\begin{enumerate}
\item $\Sigma'_d(\pi_{D,I},k)$ is always $\leq 1$. 

\item The set $S(d)$ of indices $k$ for which $\vep_d(y_k)$ equals $1$ has cardinality at most $2$. 

\item If $S(d)\not=\varnothing$ and if $j$ is its largest element, one has $\Sigma'_d(\pi_{D,I},j) = 0$ and hence 
$\Upsilon_d(\pi_{D,I},j) = 1$. 
Thus, the proposition is true for $d$ if $\vert S(d)\vert \leq 1$. 

\item If $S(d)$ contains two elements $i < j$, then 
for $k \leq j$, the value of $\Upsilon_d(\pi_{D,I},k)$ cannot increase until $k$ reaches $i$. 
\end{enumerate} 

Therefore, using Lemma \ref{lem-BS}, it suffices to prove that if $S(d)$ contains two elements $i < j$, 
then $\Sigma'_d(\pi_{D,I},i) \leq 0$, because this will imply that $\Upsilon_d(\pi_{D,I},k) \leq 1$ for $k\leq i$. 

\subsection{}\label{subsec-dead} 
Suppose that $d$ is a dead leaf such that $y_d = f_d \pi_d$. 
Then, by Lemma \ref{lem-dom}, $i = d^-$ exists and belongs to $I$. 
There are two subcases. 

\subsubsection{}\label{sub-dead-isol}  Suppose that $i\not\in D$. Then necessarily $C_I(d)$ is non-empty, hence contains a unique element 
$j$. Then, the elements $y_k$ for $k = i,d,j$ and the values of $\vep_d,\wt_d, \Sigma'_d$ and $\Upsilon_d$ 
are given by the following table: 
$$
\begin{array}{|c|c|c|c|c|c|c|c|c|c|}
\hline 
 & f_d f_i  \pi_i  &  f_d \pi_d &   f_j \pi_j 
\\
\hline 
\vep_d & 1 &  1 &  0
\\
\hline 
\wt_d & -1 & -1 & 1
\\
\hline 
\Sigma'_d & 0 & 1 & 0 
\\
\hline 
\Upsilon_d &  1  & 1  & 0
\\
\hline
\end{array}
$$
This proves the result in this case. 

\subsubsection{}\label{sub-dead-not_isol} 
Suppose now that $i\in D$. Since $y_d = f_d \pi_d$, one sees that $d$ is not an anchor dead leaf, hence 
$\sigma(i) > d$ and $m_r = \mu(i)$ is $> d$. Since $f_i$ appears in the operator $f_{L_r}$, one has 
$\wt_d(y_{m_r}) = 1$. 
Hence one obtains the following table: 
$$
\begin{array}{|c|c|c|c|}
\hline 
 & f_{\DL(i)} f_i\pi_i &  f_d \pi_d & y_{m_r} 
\\
\hline 
\vep_d & 1 & 1 & 0 
\\
\hline 
\wt_d & -1 & -1 & 1 
\\
\hline 
\Sigma'_d & 0 &  -1 & 0 
\\
\hline 
\Upsilon_d & 1  & 0 & 0
\\
\hline
\end{array}
$$
This proves the result in this case. 

\subsection{} Suppose that $d$ is a vivid leaf of $D$ and $d = \ell _r$ for some $r = 1,\dots, t$. 
This means that $C_I(d)$ has cardinality $\geq 2$, that its largest element $m$ is 
some anchor point $m_r$, and that $L_r^* = \{d\}$. (Thus, if $d^-$ exists, it does not belong to 
$D$). Then $y_m = f_d f_m \pi_m$. 

Let $i_1 < \cdots < i_q$ be the others elements of $C_I(d)$. For $s= 1,\dots, q$, 
one has $y_{i_s} = f_{i_s} \pi_{i_s}$, hence $\vep_d(y_{i_s}) = 0$ while $\wt_d(y_{i_s})=1$. 
Therefore, one obtains the following table: 
$$
\begin{array}{|c|c|c|c|c|c|c|c|c|c|}
\hline 
 & f_d \pi_d &  f_{i_1} \pi_{i_1} & \cdots & f_{i_q} \pi_{i_q} & f_{d} f_m \pi_m
\\
\hline 
\vep_d & 1 &   0 & \cdots & 0 &  1  
\\
\hline 
\wt_d & -1 & 1 & \cdots & 1 & -1 
\\
\hline 
\Sigma'_d & 1{-}q & 2{-}q & \cdots  & 1 & 0
\\
\hline 
\Upsilon_d &  2{-}q & 2{-}q & \cdots  & 1 & 1 
\\
\hline
\end{array}
$$
Since $q\geq 1$, this proves the result in this case. 

\subsection{} Now, let $d\in D$ such that $\dep_D(d) \geq 1$ and $\vert S(d) \vert =2$. 
The hypothesis $\vert S(d) \vert = 2$ implies, firstly, that $y_d = f_d \pi_d$, \ie $d$ is not a receiving node.  
It follows that $C_I(d)$ contains no dead leaf; in particular $\sigma(d)$ either is a vivid leaf or satisfies 
$\dep_D \, \sigma(d) \geq 1$. In both cases, setting $m_r = \mu(d)$, one has $m_r > \sigma(d)$. 
Secondly, the hypothesis  $\vert S(d) \vert = 2$ implies that $d = \ell_r$. 

Finally, note that $y_{\sigma(d)}$ equals $f_{\DL(\sigma(d))} f_{\sigma(d)} \pi_{\sigma(d)}$ (where 
$\DL(\sigma(d))$ can be empty), so that $\wt_d(y_{\sigma(d)} ) = 1$. And, as mentioned earlier, 
 one has $\vep_d(y_{\sigma(d)}) = 0$, by \cite{BG25}, \Prop 2.3. Therefore, one obtains the following table: 
$$
\begin{array}{|c|c|c|c|c|c|c|c|c|c|}
\hline 
 & f_d \pi_d & \cdots & y_{\sigma(d)} &   \cdots & y_{m_r} 
\\
\hline 
\vep_d & 1 &  \cdots & 0 & \cdots & 1  
\\
\hline 
\wt_d & -1 & \cdots & 1 & \cdots & -1 
\\
\hline 
\Sigma'_d & \leq 0 & \cdots & 1 & \cdots & 0
\\
\hline 
\Upsilon_d &  \leq 1 & \cdots & 1 & \cdots & 1 
\\
\hline
\end{array}
$$
This proves the result in this case. (Note that the exact values of $\Sigma'_d(\pi_{D,I}, d)$ and 
$\Upsilon_d(\pi_{D,I}, d)$ are $1{-}q$ and $2{-}q$, where $q = c_I(d)$.) 
This completes the proof of \Prop \ref{prop-2} and, therefore, of \Th \ref{th-1}.

\let\om\omega

\section{Examples and concluding remarks} 

Let us conclude with some examples and remarks. Suppose for simplicity that $\Delta = I = A_n$, with its usual numbering. 

For integers $a\leq b$ in $[1,n]$, denote by $\alpha_{[a,b]}$ the root $\alpha_a + \cdots + \alpha_b$. 
Let $W = S_{n+1}$ be the Weyl group and, for $i=1,\dots, n$, let $s_i$ be the simple reflection $s_{\alpha_i}$, \ie the transposition $(i,i{+}1)$. If, for example, $w  = s_i s_j s_k s_\ell$ is a reduced expression, we will denote it simply by 
$s_{ijk\ell}$. 

\smallskip 
Let $\pi = \pi_\rho$ be the element $\pi_1\otimes \cdots \otimes \pi_n$ in the crystal 
$B_1 \otimes \cdots \otimes B_n$.

\subsection{} Suppose that $n=3$ and $D=I$. Set $\lambda_3 = \lambda_{I,I} = 2\rho - 2\alpha_{[1,3]} = 2\omega_2$. 
Taking the vertex $1$ as the root, our element $\pi_{I,I}(\calR_1)$ is 
\begin{equation} 
f_1 f_2 f_3^2 f_2 f_1 \pi = f_1 \pi_1 \otimes f_3 f_2 \pi_2 \otimes f_1 f_2 f_3 \pi_3 . 
\end{equation} 
Using the Stembridge relations (see \cite{St03}, \Def 1.1 and \Th 2.4, or \cite{BS17}, \S\S 4.2 and 5.7)), one sees that it equals 
$$
f_1 f_3 f_2^2 f_3 f_1 \pi = f_3 f_1 f_2^2 f_1 f_3 \pi =  f_3 f_2 f_1^2 f_2 f_3 \pi , 
$$
which is our element $\pi_{I,I}(\calR_3)$ obtained by taking the vertex $3$ as the root, so the latter does not produce a new element. 

\smallskip However, taking $2$ as the root, we obtain the element $\pi_{I,I}(\calR_2) = f_2 f_1^2 f_3^2 f_2 \pi$, 
which is different. Indeed, an easy computation shows that 
\begin{equation} 
f_2 f_1^2 f_3^2 f_2 \pi  = f_1 \pi_1 \otimes f_1 f_3 f_2 \pi_2 \otimes f_2 f_3 \pi_3 . 
\end{equation} 
Now, note that $\lambda_3 = s_1 s_2 \rho + s_3 s_2 \rho$, hence $L(\lambda_3)$ is a PRV-component of the 
tensor product (see \cite{Ku88}, (C$_1$) for the definition).  

\smallskip 
Using  a short Sage code, we verified  that $s_{2132}$ is the unique element $w$ of 
$W$ such that $\rho + w\rho$ is $W$-conjugate to $\lambda_3$. Thus, 
Kumar's lower bound (see \cite{Ku89}, \Th 1.2 ) for the multiplicity of $L(\lambda_3)$ in $L(\rho)\otimes L(\rho)$ is $1$. 
Therefore, our elements provide one additional copy of this PRV-component. (And a Sage computation  in the 
Weyl character ring shows that indeed $c^{\lambda_3}_{\rho,\rho} = 2$.)

\subsection{} Suppose that $n=4$ and $D=I$. Set $\lambda_4 = \lambda_{I,I} = 2\rho - 2\alpha_{[1,4]} = 2\omega_2 + 2\omega_3$. 
Taking the vertex $1$ as the root, our element $\pi_{I,I}(\calR_1)$ is 
\begin{equation} 
f_1 f_2 f_3 f_4^2 f_3 f_2 f_1 \pi = f_1 \pi_1 \otimes f_2 \pi_2 \otimes f_4 f_3 \pi_3 \otimes f_1 f_2 f_3 f_4 \pi_4 . 
\end{equation} 
Again, using the Stembridge relations, one sees that it equals 
\begin{multline*} 
f_1 f_2 f_4 f_3^2 f_4 f_2 f_1 \pi = f_4 f_1 f_2 f_3^2 f_2 f_1 f_4 \pi = 
f_4 f_1 f_3 f_2^2 f_3 f_1 f_4 \pi 
\\ 
= f_4 f_3 f_2 f_1^2 f_2 f_3 f_4 \pi, 
\end{multline*} 
where the latter is our element $\pi_{I,I}(\calR_4)$ obtained by taking the vertex $4$ as the root. 

\smallskip However, taking successively $2$ or $3$ as the root, an easy computation shows that 
\begin{gather} 
f_2 f_1 f_3 f_4^2 f_3 f_1 f_2 \pi  = f_1 \pi_1 \otimes f_1 f_2 \pi_2 \otimes f_4 f_3 \pi_3  \otimes f_2 f_3 f_4 \pi_4 ; 
\\
f_3 f_4 f_2 f_1^2 f_2 f_4 f_3\pi = f_1 \pi_1 \otimes f_2 \pi_2 \otimes f_1 f_2 f_4 f_3 \pi_3  \otimes f_3 f_4 \pi_4 .
\end{gather} 
So, we obtain $3$ different elements. 

\smallskip 
Now, note that $\lambda_4 = s_{421} \rho + s_{134} \rho$, hence $L(\lambda_4)$ is a PRV-component. 
Using again our Sage code, we verified that $s_{121}s_{434}$ and $s_{434} s_{121}$ 
are the unique elements $w$ of $W$ such that $\rho + w\rho$ is $W$-conjugate to $\lambda_4$. Thus, 
Kumar's lower bound for $c^{\lambda_4}_{\rho,\rho}$ is $2$ and 
our elements provide one additional copy of this PRV-component. 

On the other hand, 
a Sage Weyl character ring computation shows that $c^{\lambda_4}_{\rho,\rho} = 7$; therefore, 
it would be interesting to find other $\rho$-dominant elements of $B(\rho)$ of weight $\lambda_4 - \rho$.

\subsection{} Suppose that $n=4$ and $D = \{2,3\}$. Set 
\begin{equation} 
\mu_4 = \lambda_{D,I} = 2\rho -\alpha_{[1,4]} - \alpha_{[2,3]} = 2 \omega_1 + \omega_2 + \omega_3 + 2\omega_4.
\end{equation} 
Firstly, by \cite{BG25}, \Th 2.2, the elements 
$f_1 f_2^2 f_3^3 f_4 \pi$ and $f_4 f_3^2 f_2^2 f_1 \pi$ are $\rho$-dominant elements 
of $B(\rho)$. An easy computation shows that 
\begin{gather} 
f_1 f_2^2 f_3^2 f_4 \pi = \pi_1 \otimes \pi_2 \otimes f_2 f_3 \pi_3 \otimes f_1 f_2 f_3 f_4 \pi_4 , 
\\
f_4 f_3^2 f_2^2 f_1 \pi = f_2 f_1 \pi_1 \otimes f_3 f_2 \pi_2 \otimes f_3 \pi_3 \otimes f_4 \pi_4 , 
\end{gather}
hence these elements are distinct. Further, for the  four possible choices of the root, we 
obtain four elements $\pi_{D, I}(\calR_i)$. Taking 1 as the root, our Proposition \ref{prop-1} 
tells us that 
\begin{equation} 
f_2 f_3 f_4 f_3 f_2 f_1  \pi = f_1 \pi_1 \otimes f_3 f_2 \pi_2 \otimes f_2 f_3 \pi_3 \otimes f_4 \pi_4 . 
\end{equation}
Taking successively $2, 3$ and $4$ as the root, an easy computation shows that 
\begin{align} 
f_2 f_3 f_4 f_1 f_3 f_2 \pi &= \pi_1 \otimes f_1 f_3 f_2\pi_2 \otimes f_2 f_3 \pi_3 \otimes  f_4 \pi_4 ; 
\\
f_3 f_2 f_1 f_4 f_2 f_3 \pi &= \pi_1 \otimes f_2 \pi_2 \otimes f_1 f_2 f_3 \pi_3 \otimes f_3 f_4 \pi_4 ;
\\
f_3 f_2 f_1 f_2 f_3 f_4 \pi &= \pi_1 \otimes f_2 \pi_2 \otimes f_3 \pi_3 \otimes f_1 f_2  f_3 f_4 \pi_4 . 
\end{align}
So all six elements are distinct. Again, $L(\mu_4)$ is a PRV-component since one has, 
for example, 
$$
\mu_4 = s_2 s_1 \rho + s_3 s_4 \rho = s_2 s_1 (\rho + s_1 s_2 s_3 s_4 \rho). 
$$
Again, using our Sage code, we verified that $\rho + w\rho$ is $W$-conjugate to $\mu_4$ if and only if $w = 
s_{1234}, s_{4321}, s_{23214}$, or $s_{14232}$, hence Kumar's lower bound is $4$. 
Thus, our results combined with \cite{BG25}, \Th 2.2 provide two additional copies of this PRV-component. 

On the other hand, a Sage Weyl character ring computation shows that $c_{\rho,\rho}^{\mu_4} = 12$. 
Therefore it would be interesting to find other families of $\rho$-dominant elements of $B(\rho)$ of weight 
$\mu_4 - \rho$.

\end{document}